\documentclass[reqno,onecolumn,12pt]{amsart}
\usepackage{amsfonts}
\usepackage{amssymb}
\usepackage{amsmath}
\usepackage{graphicx}
\usepackage{amscd,amsthm}

\newtheorem{theorem}{Theorem}
\theoremstyle{plain}

\newtheorem{corollary}{Corollary}

\newtheorem{example}{Example}

\numberwithin{equation}{section}

\input{tcilatex}

\begin{document}
\author{}
\title{}
\maketitle

\begin{center}
\pagestyle{myheadings} \thispagestyle{empty} 
\markboth{\bf Mehmet Acikgoz, Ilknur Koca, Serkan Araci}
{\bf The evaluation of the sum of more general series by Bernstein Polynomials}

\textbf{\large THE EVALUATION OF THE SUMS OF MORE GENERAL SERIES BY
BERNSTEIN POLYNOMIALS}

\bigskip

\textbf{Mehmet Acikgoz$^{\ast }$, Ilknur Koca and Serkan Araci}

\bigskip University of Gaziantep, Faculty of Science and Arts, Department of
Mathematics, 27310 Gaziantep, TURKEY

\hspace{0.5cm}

\textbf{E-Mails: acikgoz@gantep.edu.tr; ibaltaci@gantep.edu.tr;
mtsrkn@gmail.com}

\textbf{$^{\ast }$Corresponding Author}\\[2mm]

\hspace{0.5cm}

\textbf{\large Abstract}

\hspace{0.5cm}
\end{center}

\begin{quotation}
Let $n,k$ be the positive integers, and let $S_{k}\left( n\right) $ be the
sums of the $k$-th power of positive integers up to $n$: $S_{k}\left(
n\right) =\sum_{l=1}^{n}l^{k}$. By means of that we consider the evaluation
of the sum of more general series by Bernstein polynomials. Additionally we
show the reality of our idea with some examples.

\hspace{0.5cm}
\end{quotation}



\section{\textbf{Introduction}}

The history of Bernstein polynomials depends on Bernstein in 1904. It is
well known that Bernstein polynomials play a crucial important role in the
area of approximation theory and the other areas of mathematics, on which
they have been studied by many researchers for a long time [1, 3, 5-7, 10,
11, 16, 17]. These polynomials also take an important role in physics.

Recently the works including applications of umbral calculus to Genocchi
numbers and polynomials \cite{Araci2}, the Legendre polynomials associated
with Bernoulli, Euler, Hermite and Bernstein polynomials \cite%
{AraciAcikgoz-Bagdasaryan-Sen}, the applications of umbral calculus to
extended Kim's $p$-adic $q$-deformed fermionic integrals in the $p$-adic
integer ring \cite{Araci-Acikgoz-Sen}, the integral of the product of
several Bernstein polynomials \cite{Acikgoz-Araci1}, the generating function
of Bernstein polynomials \cite{Acikgoz-Araci2}, a theorem concerning
Bernstein polynomials \cite{Gould}, new generating function of the ($q$-)
Bernstein type polynomials and their interpolation function \cite%
{Simsek-Acikgoz}, $q$-analogues of the sums of powers of consecutive
integers, squares, cubes, quarts and quints [12-15, 18-20], have been
investigated extensively.

In the complex plane, the Bernoulli polynomials $B_{n}\left( x\right) $ are
known by the following generating series:%
\begin{equation}
\sum_{n=0}^{\infty }B_{n}\left( x\right) \frac{t^{n}}{n!}=\frac{t}{e^{t}-1}%
e^{xt}\text{, }\left\vert t\right\vert <2\pi \text{.}  \label{Eq. 1}
\end{equation}

In the case $x=0$ in (\ref{Eq. 1}), we have $B_{n}\left( 0\right) :=B_{n}$
that stands for Bernoulli numbers. By (\ref{Eq. 1}), we have%
\begin{equation}
B_{n}\left( x\right) =\dsum\limits_{k=0}^{n}\left( 
\begin{array}{c}
n \\ 
k%
\end{array}%
\right) B_{k}x^{n-k}\text{.}  \label{Eq. 2}
\end{equation}

The Bernoulli numbers satisfy the following identity%
\begin{equation*}
B_{0}=1\text{ and }\left( B+1\right) ^{n}-B_{n}=\delta _{1,n}
\end{equation*}%
where $\delta _{1,n}$ stands for Kronecker's delta and we have used $%
B^{n}:=B_{n}$ (for details, see \cite{AraciAcikgoz-Bagdasaryan-Sen}, \cite%
{Acikgoz-Araci3}, \cite{Cheon}, \cite{Kim-Kim-Lee-Ryoo}).

Recently, Acikgoz and Araci has constructed the generating function for the
Bernstein polynomials $B_{k,n}\left( x\right) $ by the following rule: 
\begin{equation}
\sum_{n=k}^{\infty }B_{k,n}\left( x\right) \frac{t^{n}}{n!}=\frac{\left(
tx\right) ^{k}}{k!}e^{t\left( 1-x\right) }\text{ }\left( t\in 
\mathbb{C}
\text{ and }k=0,1,2,\cdots ,n\right) \text{.}  \label{Eq. 3}
\end{equation}

By (\ref{Eq. 3}), we see that%
\begin{equation*}
\sum_{n=k}^{\infty }B_{k,n}\left( x\right) \frac{t^{n}}{n!}%
=\sum_{n=k}^{\infty }\left( \left( 
\begin{array}{c}
n \\ 
k%
\end{array}%
\right) x^{k}\left( 1-x\right) ^{n-k}\right) \frac{t^{n}}{n!}
\end{equation*}%
by comparing the coefficients of $\frac{t^{n}}{n!}$ in the above, we derive
well known expression of Bernstein polynomials, as follows: For $k,n\in
Z_{+} $ 
\begin{equation}
B_{k,n}\left( x\right) =\left( 
\begin{array}{c}
n \\ 
k%
\end{array}%
\right) x^{k}\left( 1-x\right) ^{n-k}  \label{Eq. 4}
\end{equation}
where $x\in \left[ 0,1\right] $ and $\left( 
\begin{array}{c}
n \\ 
k%
\end{array}%
\right) $ is known as

\begin{equation*}
\left( 
\begin{array}{c}
n \\ 
k%
\end{array}%
\right) =\left\{ 
\begin{array}{ccc}
\frac{n!}{k!\left( n-k\right) !} & , & \text{if }n\geq k\text{ } \\ 
0 & , & \text{if }n<k.%
\end{array}%
\right. .
\end{equation*}

It follows from (\ref{Eq. 4}) that a few Bernstein polynomials are as
follows:%
\begin{eqnarray*}
B_{0,0}\left( x\right) &=&1,B_{0,1}\left( x\right) =1-x,B_{1,1}\left(
x\right) =x,B_{0,2}\left( x\right) =(1-x)^{2},B_{1,2}\left( x\right)
=2x\left( 1-x\right) \\
B_{2,2}\left( x\right) &=&x^{2},B_{0,3}\left( x\right) =\left( 1-x\right)
^{3},B_{1,3}\left( x\right) =3x\left( 1-x\right) ^{2},B_{2,3}\left( x\right)
=3x^{2}\left( 1-x\right) ,B_{3,3}\left( x\right) =x^{3}\text{.}
\end{eqnarray*}

In the same time, the Bernstein polynomials $B_{k,n}\left( x\right) $ have
several properties of interest:

\begin{itemize}
\item $B_{k,n}\left( x\right) \geq 0$, for $0\leq x\leq 1$ and\ $k=0,1,...,n$

\item Bernstein polynomials have the symmetry property $B_{k,n}\left(
x\right) =B_{n-k,n}\left( 1-x\right) $

\item $\dsum_{k=0}^{n}B_{k,n}\left( x\right) =1$, which is know a part of
unity.

\item $B_{k,n}\left( x\right) =(1-x)B_{k,n-1}\left( x\right)
+xB_{k-1,n-1}\left( x\right) $ with $B_{k,n}\left( x\right) =0$ for $k<0,$ $%
k>n$ and $B_{0,0}\left( x\right) =1$ \textit{cf.} \cite{Araci1}, \cite%
{AraciAcikgoz-Bagdasaryan-Sen}, \cite{Acikgoz-Araci1}, \cite{Acikgoz-Araci2}%
, \cite{Acikgoz-Araci3}, \cite{Gould}, \cite{Kim3}, \cite{Kim-Kim-Lee-Ryoo}.
\end{itemize}

From (\ref{Eq. 1}), a few Bernoulli polynomials can be generated as 
\begin{equation*}
B_{0}\left( x\right) =1,B_{1}\left( x\right) =x-\frac{1}{2},B_{2}\left(
x\right) =x^{2}-x+\frac{1}{6},B_{3}\left( x\right) =x^{3}-\frac{3}{2}x^{2}+%
\frac{1}{2}x\text{.}
\end{equation*}

For any positive integer $n$, followings are the most known first three sums
of powers of integers:%
\begin{equation*}
1+2+3+...+n=\frac{n(n+1)}{2},
\end{equation*}%
\begin{equation*}
1^{2}+2^{2}+3^{2}+...+n^{2}=\frac{n(n+1)(2n+1)}{6}
\end{equation*}%
and%
\begin{equation*}
1^{3}+2^{3}+3^{3}+...+n^{3}=\left( 1+2+3+...+n\right) ^{2}=\left[ \frac{%
n(n+1)}{2}\right] ^{2}\text{.}
\end{equation*}

Formulas for sums of integer powers were first given in generalizable form
by mathematician Thomas Harriot (c. 1560-1621) of England. At about the same
time, Johann Faulhaber (1580-1635) of Germany gave formulas for these sums,
but he did not make clear how to generalize them. Also Pierre de Fermat
(1601-1665) and Blaise Pascal (1623-1662) gave the formulas for sums of
powers of integers.

The Swiss mathematician Jacob Bernoulli (1654-1705) is perhaps best and most
deservedly known for presenting formulas for sums of integer powers. Because
he gave the most explicit sufficient instructions for finding the
coefficients of the formulas [12-15, 18-20].

So, we are interested in finding a method to derive a formula for the sums
of powers of integers. Following an idea due to J. Bernoulli, we aim to
obtain a Theorem which gives the method for the evaluation of the sums of
more general series by Bernstein polynomials.

\section{\textbf{The Evaluation of the Sums of More General Series by
Bernstein Polynomials}}

In the 17\textit{th} century a topic of mathematical interest was finite
sums of power of integers such as the series $1+2+3+\cdots +\left(
n-1\right) $ or the series $1^{2}+2^{2}+3^{2}+\cdots +\left( n-1\right) ^{2}$%
. The closed form for these finite sums were known, but the sums of the more
general series $1^{k}+2^{k}+3^{k}+...+\left( n-1\right) ^{k}$ was not. It
was the mathematician Jacob Bernoulli who would solve this problem with the
following equality [12-15, 18-20]. The sum of the $k$-th powers of the first 
$\left( n-1\right) $ integers is given by the formula%
\begin{equation}
1^{k}+2^{k}+3^{k}+...+\left( n-1\right) ^{k}=\dint\limits_{1}^{n}B_{k}\left(
x\right) dx  \label{Eq. 5}
\end{equation}%
using the integral of the Bernoulli polynomials, $B_{n}\left( x\right) ,$
under integral from $1$ to $n$.

We are now in a position to express our aim as Theorem 1 for the evaluation
of the sum of more general series by Bernstein polynomials.

\begin{theorem}
\label{Thm1}Let $n,$ $k$ and $m$ be positive integers and let $S_{m}\left(
n\right) $ be $\sum_{l=1}^{n}l^{m}$, then we have%
\begin{equation*}
S_{m}\left( n\right) =\frac{\left( -n^{-1}\right) ^{k}}{\left( m+k+1\right) !%
}\sum_{l=k}^{m+k+1}\binom{m+k+1}{l}B_{m+k-l+1}B_{k,l}\left( -n\right) -\frac{%
1}{\left( m+1\right) !}\sum_{l=0}^{m+1}\binom{m+1}{l}2^{m+1-l}B_{l}+1\text{.}
\end{equation*}
\end{theorem}

\begin{proof}
To prove this Theorem, we take $\sum_{k=0}^{\infty }\frac{t^{k}}{k!}$ in the
both sides of the Eq. (\ref{Eq. 5}), so it yields to%
\begin{eqnarray*}
e^{t}+e^{2t}+\cdots +e^{\left( n-1\right) t} &=&\dint\limits_{1}^{n}\left(
\sum_{k=0}^{\infty }B_{k}\left( x\right) \frac{t^{k}}{k!}\right) dx \\
&=&\dint\limits_{1}^{n}\left[ \frac{t}{e^{t}-1}e^{xt}\right] dx \\
&=&\left[ \sum_{m=0}^{\infty }B_{m}\frac{t^{m-1}}{m!}\right] \left[
e^{nt}-e^{t}\right] \\
&=&\left[ \sum_{m=0}^{\infty }B_{m}\frac{t^{m-1}}{m!}\right] \left[ \frac{%
n^{-k}\left( -1\right) ^{k}k!}{e^{t}}\sum_{m=k}^{\infty }B_{k,m}\left(
-n\right) \frac{t^{m-k}}{m!}-e^{t}\right]
\end{eqnarray*}%
from the last identity, we see that 
\begin{equation}
e^{2t}+e^{3t}+\cdots +e^{nt}=\frac{1}{t}\left[ \sum_{m=0}^{\infty }B_{m}%
\frac{t^{m}}{m!}\right] \left[ \frac{n^{-k}\left( -1\right) ^{k}k!}{t^{k}}%
\sum_{m=0}^{\infty }B_{k,m}\left( -n\right) \frac{t^{m}}{m!}%
-\sum_{m=0}^{\infty }2^{m}\frac{t^{m}}{m!}\right]  \label{Eq. 6}
\end{equation}%
by using Cauchy product rule in the right hand side of Eq. (\ref{Eq. 6}), we
have%
\begin{equation*}
I_{1}=\sum_{m=0}^{\infty }\left( n^{-k}\left( -1\right) ^{k}k!\sum_{l=k}^{m}%
\binom{m}{l}B_{m-l}B_{k,l}\left( -n\right) \right) \frac{t^{m-k-1}}{m!}%
-\sum_{m=0}^{\infty }\left( \sum_{l=0}^{m}\binom{m}{l}2^{m-l}B_{l}\right) 
\frac{t^{m-1}}{m!}\text{.}
\end{equation*}%
By (\ref{Eq. 6}), we derive the following 
\begin{equation*}
I_{2}=\sum_{m=0}^{\infty }\left( 2^{m}+3^{m}+\cdots +n^{m}\right) \frac{t^{m}%
}{m!}.
\end{equation*}%
When we equate $I_{1}$ and $I_{2}$, we have%
\begin{eqnarray*}
1^{m}+2^{m}+3^{m}+\cdots +n^{m} &=&\frac{\left( -n^{-1}\right) ^{k}}{\left(
m+k+1\right) !}\sum_{l=k}^{m+k+1}\binom{m+k+1}{l}B_{m+k-l+1}B_{k,l}\left(
-n\right) \\
&&-\frac{1}{\left( m+1\right) !}\sum_{l=0}^{m+1}\binom{m+1}{l}%
2^{m+1-l}B_{l}+1\text{.}
\end{eqnarray*}%
Thus, we complete the proof of the Theorem.
\end{proof}

Let $m=k$ in Theorem \ref{Thm1}, we readily get the following Corollary $1$.

\begin{corollary}
\label{Corollary1}Let $n$ and $k$ be positive integers and let $S_{k}\left(
n\right) $ be $\sum_{l=1}^{n}l^{k}$, then we have%
\begin{equation*}
S_{k}\left( n\right) =\frac{\left( -n^{-1}\right) ^{k}}{\left( 2k+1\right) !}%
\sum_{l=k}^{2k+1}\binom{2k+1}{l}B_{2k-l+1}B_{k,l}\left( -n\right) -\frac{1}{%
\left( k+1\right) !}\sum_{l=0}^{k+1}\binom{k+1}{l}2^{k+1-l}B_{l}+1.
\end{equation*}
\end{corollary}

\begin{example}
Taking $k=1$ in Corollary \ref{Corollary1}, we see that%
\begin{gather*}
1+2+3+...+n=\frac{-n^{-1}}{6}\sum_{l=1}^{3}\binom{3}{l}B_{3-l}B_{1,l}\left(
-n\right) -\frac{1}{2}\sum_{l=0}^{2}\binom{2}{l}2^{2-l}B_{l}+1 \\
=\frac{n\left( n+1\right) }{2}\text{.}
\end{gather*}%
For $k=2$ in Corollary \ref{Corollary1}, we have%
\begin{gather*}
1^{2}+2^{2}+3^{2}+...+n^{2}=\frac{n^{-2}}{120}\sum_{l=2}^{5}\binom{5}{l}%
B_{5-l}B_{2,l}\left( -n\right) -\frac{1}{6}\sum_{l=0}^{3}\binom{3}{l}%
2^{3-l}B_{l}+1 \\
=\frac{n\left( n+1\right) \left( 2n+1\right) }{6}\text{.}
\end{gather*}%
By similar way, it can be easily shown for $k=3,4,\cdots $.
\end{example}

\end{document}